\documentclass[a4paper,11pt]{amsart}

\usepackage{latexsym}
\usepackage{amssymb,amsmath}
\usepackage{graphics}
\usepackage{url}
\usepackage[vcentermath]{youngtab}

\usepackage[margin=0.0cm]{caption}
\usepackage{booktabs}

\newtheorem{theorem}{Theorem}

\newtheorem{proposition}[theorem]{Proposition}

\newtheorem{lemma}[theorem]{Lemma}

\theoremstyle{definition}
\newtheorem{definition}[theorem]{Definition}

\newcommand{\N}{\mathbf{N}}

\renewcommand{\epsilon}{\varepsilon}

\DeclareMathOperator{\sgn}{sgn}

\newcommand{\wb}{\circ}
\newcommand{\bb}{\bullet}

\newcounter{thmlistcnt}
\newenvironment{thmlist}%
	{\setcounter{thmlistcnt}{0}%
	\begin{list}{\emph{(\roman{thmlistcnt})}}{%
		\usecounter{thmlistcnt}%
		\setlength{\topsep}{0pt}%
		\setlength{\leftmargin}{24pt}%
		\setlength{\itemsep}{0pt}%
		\setlength{\labelwidth}{17pt}
		\setlength{\itemindent}{0pt}}%
	}%
	{\end{list}}%

\newcounter{caseslistcnt}
	{\setcounter{caseslistcnt}{0}%
	\begin{list}{(\roman{thmlistcnt})}{%
		\usecounter{thmlistcnt}%
		\setlength{\topsep}{0pt}%
		\setlength{\leftmargin}{48pt}%
		\setlength{\itemsep}{0pt}%
		\setlength{\labelwidth}{17pt}
		\setlength{\itemindent}{0pt}}%
	}%
	{\end{list}}%

\newcounter{typelistcnt}
	{\setcounter{typelistcnt}{0}%
	\begin{list}{}{
		\usecounter{typelistcnt}%
		\setlength{\topsep}{0pt}%
		\setlength{\leftmargin}{55pt}%
		\setlength{\itemsep}{0pt}%
		\setlength{\labelwidth}{55pt}
		\setlength{\itemindent}{00pt}}%
	}%
	{\end{list}}%

\newcommand{\vvdash}{\,\vdash\,}


\begin{document}
\title[A plethystic Murnaghan--Nakayama rule]{A combinatorial proof of a plethystic 
Murnaghan--Nakayama rule}
\date{\today}

\maketitle
\thispagestyle{empty}

\section{Introduction}

The purpose of this note is to give a combinatorial 
proof of a  plethystic generalization of the Murnaghan--Nakayama rule,
first stated in \cite{DLT}. The key step in the proof uses a  sign-reversing pairing
on sequences of bead moves on James' abacus (see \cite[page 78]{JK}), inspired by
the theme of \cite{LoehrAbacus} `when beads bump, objects cancel'.
The only prerequisites are the Murnaghan--Nakayama rule and basic facts about plethysms of symmetric functions.
The necessary combinatorial background on border-strips and James' abacus is recalled in Section~2 below.

Let $s_{\lambda / \nu}$ denote the Schur function corresponding to the skew-partition~$\lambda / \nu$ and
let $p_r$ denote the power-sum symmetric function of degree \hbox{$r \in \N$}. Let
$\sgn(\lambda / \nu) = (-1)^\ell$ if $\lambda / \nu$ is a border-strip of height $\ell \in \N_0$, and
let $\sgn(\lambda / \nu) = 0$ otherwise.
The Murnaghan--Nakayama rule (see, for instance, \cite[Theorem 7.17.1]{StanleyII}) states that if $\nu$
is a partition and $r \in \N$ then
\begin{equation}
\label{eq:MN} s_\nu p_r = \sum_{\lambda \vvdash r + |\nu|} \sgn(\lambda / \nu) s_\lambda.
\end{equation}

To generalize~\eqref{eq:MN} we need some further definitions. Let $\lambda / \nu$ be a skew-partition
and let $d$ be minimal such that $\lambda_d > \nu_d$.
We say that an $r$-border-strip $\lambda / \mu$ is the \emph{final} $r$-border-strip in $\lambda / \nu$
if $\mu / \nu$ is a skew-partition and 
$\lambda_d > \mu_d$. 
Thus $\lambda / \mu$ has a (necessarily unique) final $r$-border-strip if and only if
the~$r$ boxes at the top-right of the rim of the Young diagram of $\lambda / \nu$
can be removed to leave a skew-partition.
We say that
$\lambda / \nu$ is \emph{$r$-decomposable} if there exist partitions
$\mu^{(0)}, \mu^{(1)}, \ldots, \mu^{(m)}$ such that
\[ \lambda = \mu^{(0)} \supset \mu^{(1)} \supset \ldots \supset \mu^{(m)} = \nu \]
and $\mu^{(i)}/\mu^{(i+1)}$ is the final $r$-border-strip in $\mu^{(i)} / \nu$ for each $i$.
In this case we define
$\sgn_r(\lambda / \nu) = \sgn(\mu^{(0)}/\mu^{(1)}) \ldots \sgn(\mu^{(m-1)} / \mu^{(m)})$. If $\lambda / \nu$
is not $r$-decomposable we define $\sgn_r(\lambda / \mu) = 0$.
Let $f \circ g$ denote the plethysm of symmetric functions $f$ and $g$, as defined in 
\cite[I.8]{MacDonald} or \cite[Appendix~2]{StanleyII}.
Finally let $h_m = s_{(m)}$ denote the complete symmetric function of degree $m \in \N_0$.

We shall prove that if $\nu$ is a partition and $r$, $m \in \N$ then
\begin{equation}\label{eq:main}
s_\nu (p_r \circ h_m) = \sum_{\lambda \vvdash rm + |\nu|} \sgn_r(\lambda / \nu) s_\lambda.
\end{equation}
Taking $m=1$ recovers~\eqref{eq:MN}.
The formula for $s_\mu (p_r \circ h_{m_1}\ldots h_{m_d})$ given in \cite[page 29]{DLT}
follows by repeated applications of~\eqref{eq:main}.
This formula is proved in \cite{DLT}
using Muir's rule. Similarly~\eqref{eq:main} implies combinatorial formulae for $s_\mu (p_{r_1}\ldots p_{r_c} \circ h_m)$,
and, more generally, for  $s_\mu (p_{r_1}\ldots p_{r_c} \circ h_{m_1}\ldots h_{m_d})$.
 An alternative proof of~\eqref{eq:main} using the character theory of the symmetric group
was given in \cite[Proposition 4.3]{EPW}. The special case $\nu = \varnothing$ of~\eqref{eq:main} follows from
\cite[I.8, Example 8]{MacDonald}.

\section{Background on border-strips, signs and James' abacus}


Let $\lambda$ be a partition of $n$ with $p$ parts. The \emph{Young diagram} of $\lambda$ is the set
\[ [\lambda] = \{ (i,j) : 1 \le i \le p,\ 1 \le j \le \lambda_i \}. \]
The elements of a Young diagram are called \emph{boxes}. 
The \emph{rim} of $\lambda$ consists of all boxes
$(i,j)$ such that $(i+1,j+1)\not\in [\lambda]$.
A \emph{border-strip} of length $s$ in $\lambda$ consists of $s$ adjacent boxes 
in the rim of $\lambda$
whose removal from $[\lambda]$ leaves the Young diagram of a partition. 
The \emph{top-right} and \emph{bottom-left} boxes of a border-strip are defined
with respect to the `English' convention for drawing Young diagrams,  shown in
Figure~1 below.
The \emph{height} of a border-strip with bottom-left
box $(i,j)$ and top-right box $(i',j')$ is $i-i'$. 

\begin{figure}[b]
\phantom{tf}
\vspace*{-30pt}
\begin{center}
\hspace*{-12pt}\includegraphics{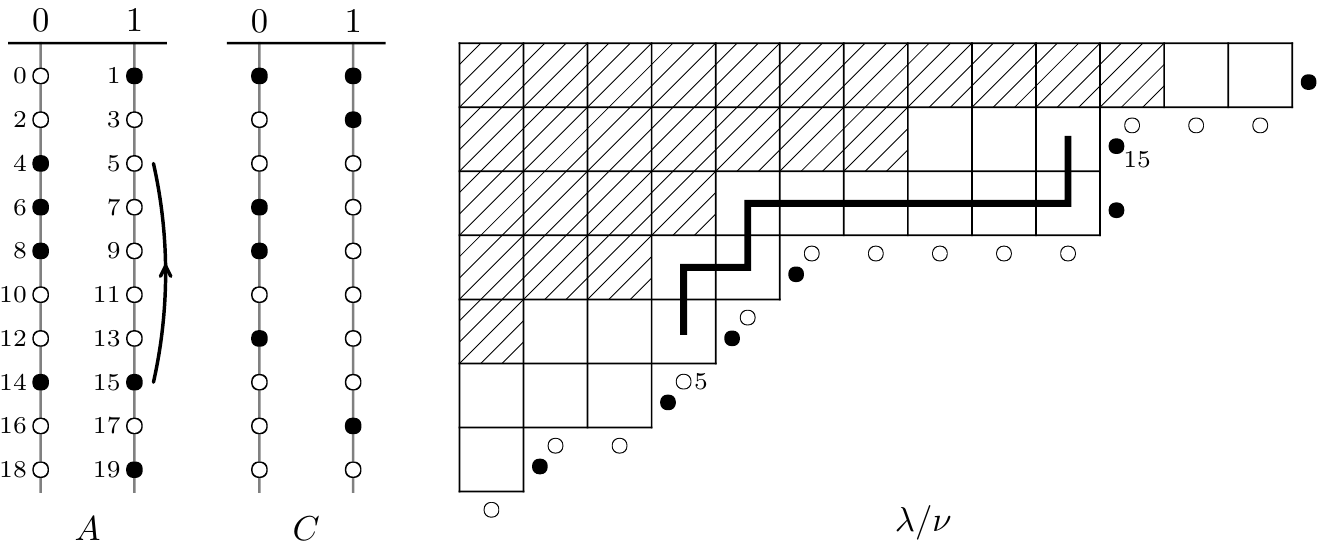}
\end{center}
\caption{\small The $2$-decomposable skew-partition $\lambda / \nu$ where
$\lambda = (13,10,10,5,4,3,1)$ and $\nu = (11,7,4,3,1)$. The normalized abacus $A$
for $\lambda$ and an abacus $C$ for $\nu$ are shown on two runners. The
marked $10$ border-strip of height $3$ has top-right box $(2,10)$ and bottom-left box
$(5,4)$; it
 corresponds to the bead in position $15$. Swapping this bead with the gap in position $5$
removes this border-strip, giving  $\mu = (13,9,4,3,3,3,1)$.
In the walk along
the rim of $\mu$, step $5$ is up $\bb$, rather that right $\wb$; the walk then agrees with that for $\lambda$
until step $15$, which is right $\wb$, rather than up $\bb$.}
\end{figure}

Clearly $\lambda$ is determined by
the sequence of right and up steps that starts at $(p,1)$, 
visits exactly the boxes in the rim of $\lambda$, and finishes at $(1,\lambda_1)$. Encoding each right step by a \emph{gap}, denoted $\wb$, and each up step by a \emph{bead}, denoted $\bb$,
we obtain the \emph{normalized abacus} for $\lambda$. (This term is not entirely standard, but is convenient here.)
More generally, an \emph{abacus} for $\lambda$ is a sequence 
consisting of any number of beads, followed by the normalized abacus for $\lambda$, followed by any number of gaps. We number the positions in such a sequence
from $0$.

In the proof of~\eqref{eq:main} we shall only be concerned with border-strips whose length
is a multiple of a fixed $r \in \N$. 
In this case it is useful to represent
the gaps and beads on $r$ \emph{runners}, so that for each $t \in \{0,1,\ldots, r-1\}$ the positions
on runner $t$ are $t, t+r, t+2r, \ldots $. We say that position $t+jr$ is \emph{above}
position $t+j'r$ if $j < j'$. Define a \emph{single-step} bead move 
to be a move of a bead from a position $\beta$ to the position $\beta - r$
immediately above it. 
These definitions are illustrated in Figure 1 on the previous page.

The following two lemmas 
records some basic results on the abacus.

\begin{lemma}\label{lemma:basic}
Let $A$ be an abacus for the partition $\lambda$. Let $s \in \N$.
Let $\mathcal{A}_s$ be the set of bead positions $\beta$ of $A$ such
that $\beta - s \ge 0$ and $A$ has a gap in position $\beta -s$. 

\begin{thmlist}
\item 
The map sending $\beta \in \mathcal{A}_s$ to the corresponding box in $[\lambda]$
is a bijection between $\mathcal{A}_s$ and the top-right boxes in the $s$ border-strips in $\lambda$.

\item Let $\beta \in \mathcal{A}_s$ and let $B$ be the abacus obtained from 
$A$ by swapping the bead in position $\beta$
with the gap in position $\beta - s$. Then $B$ is an abacus for the partition 
obtained by removing the $s$ border-strip
from $\lambda$ corresponding to $\beta$.

\item
The height of the $s$ border-strip in $\lambda$ corresponding to the bead 
in position $\beta \in \mathcal{A}_s$ is
the number of beads in positions $\beta-s+1, \ldots, \beta-1$ of~$A$.
\end{thmlist}
\end{lemma}

\begin{proof}
Parts (i) and (ii) follow from Lemma 2.7.13 in \cite{JK}, using that 
the set of bead positions
in $A$ is a set of $\beta$-numbers for $\lambda$. 
(An alternative proof,
avoiding $\beta$-numbers, is indicated in the caption to Figure~1.)
For~(iii), observe that the beads in positions
$\beta-s+1, \ldots, \beta$ of $A$ encode the steps up made when walking the $s$ border-strip
in $\lambda$ corresponding to $\beta$. 
\end{proof}

\begin{lemma}\label{lemma:signs}
Let $\lambda = \mu^{(0)}, \mu^{(1)}, \ldots, \mu^{(c)} = \nu$ be a sequence of partitions
such that $\mu^{(i)} / \mu^{(i+1)}$ is an $s_i$ border-strip in $\mu^{(i)}$ for each $i \in \{0,1,\ldots, c-1\}$.
Let $A$ be an abacus for $\lambda$. Let $\mathcal{J}$ be the set of pairs
of positions 
$\{ \beta,\beta' \}$ such that 
\begin{thmlist}
\item $\beta < \beta'$;
\item $A$ has beads $b$ and $b'$ in positions $\beta$ and $\beta'$, respectively;
\item after the sequence of bead moves that removes the border-strips
$\mu^{(0)}/\mu^{(1)}$, $\ldots$, $\mu^{(c-1)}/\mu^{(c)}$,  bead $b$ finishes 
in a greater numbered position than  bead~$b'$.
\end{thmlist}
Then $\sgn(\mu^{(0)}/\mu^{(1)}) \ldots \sgn(\mu^{(c-1)}/\mu^{(c)}) = (-1)^{|\mathcal{J}|}$.
\end{lemma}

\begin{proof}
We work by induction on $c$. The base case $c=0$ is trivial. Let $\mathcal{I}$ be the
set defined in the same way as $\mathcal{J}$ for the 
sequence $\lambda = \mu^{(0)}$, $\mu^{(1)}$, $\ldots$,
$\mu^{(c-1)} = \mu$. Let $B$ be the abacus for $\mu$ obtained
from $A$ by the sequence of bead moves specified in (iii), stopping at $\mu$. 
Suppose that the border-strip $\mu / \nu$ corresponds
to the bead in position~$\gamma$ of $B$, and that his bead was in position $\beta$ of $A$.
Let $\ell$ be the height of $\mu / \nu$.
By Lemma~\ref{lemma:basic}(iii) there are $\ell$ beads in positions $\gamma-s_c+1, \ldots, \gamma-1$
of~$B$. Suppose that exactly~$j$ of these beads were originally in a position $\beta' > \beta$ of $A$.
These $j$ beads correspond to pairs $\{ \beta, \beta' \} \in \mathcal{I} \backslash \mathcal{J}$
and the remaining $\ell-j$ beads correspond to pairs $\{ \beta, \beta' \} \in \mathcal{J} \backslash \mathcal{I}$.
Apart from these pairs, the sets $\mathcal{I}$ and $\mathcal{J}$ agree. Thus
$|\mathcal{J}| = |\mathcal{I}| - j + (\ell-j) = |\mathcal{I}| + \ell - 2j$. 
Hence, by induction,
\begin{align*} (-1)^{|\mathcal{J}|} &= (-1)^{|\mathcal{I}|} (-1)^\ell \\ &= 
\sgn(\mu^{(0)}/\mu^{(1)}) \ldots \sgn(\mu^{(c-2)}/\mu^{(c-1)}) \sgn(\mu^{(c-1)}/\mu^{(c)})
\end{align*}
as required.
\end{proof}


For the remainder of this section, fix $r$, $m \in \N$ and let $\lambda/\nu$ be a skew-partition of $rm$ such that $\nu$ can be
obtained from $\lambda$ by repeatedly removing~$r$ border-strips.
Let $A$ be an $r$-runner abacus for $\lambda$ 
and let $C$ be the abacus for $\nu$ obtained from $A$ by a sequence of bead moves
that removes these $r$ border-strips. 


Using Lemma~\ref{lemma:signs} we obtain the following proposition, which is equivalent
to \cite[2.7.26]{JK}. 

\begin{proposition}\label{prop:anyorder}
Let $\lambda / \nu$ and $A$ be as just defined.
There exists $\sigma \in \{1, -1\}$ such that 
if $\mu^{(0)}$, $\mu^{(1)}$, $\ldots$, $\mu^{(m)}$
is any sequence of partitions such that $\mu^{(0)} = \lambda$, $\mu^{(m)} = \nu$ 
and $\mu^{(i)}/\mu^{(i+1)}$ is an $r$ border-strip
in $\mu^{(i)}$ for each $i \in \{0,1,\ldots, m-1\}$, then
\[ \sigma = \sgn(\mu^{(0)}/\mu^{(1)}) \ldots \sgn(\mu^{(m-1)} / \mu^{(m)}).
\]
\end{proposition}

\begin{proof}
The sequence $\mu^{(0)}$, 
$\mu^{(1)}$, $\ldots$, $\mu^{(m)}$ corresponds to a sequence of single-step bead moves
on $A$ leading to the abacus $C$.
Since the final positions of the beads moved on $A$ are independent
of the order of moves, the result
follows from Lemma~\ref{lemma:signs}.
\end{proof}

An immediate corollary of Proposition~\ref{prop:anyorder} is that $\sgn_r(\lambda/\nu)$ 
can be computed by removing $r$ border-strips in any way.
We use this corollary in the proof of
Proposition~\ref{prop:I}  below.

We end this section with a  characterization of $r$-decomposable
partitions using the abacus.

\begin{definition}\label{defn:rdec} Let the abaci $A$ and $C$ be as defined.
Let $t \in \{0,\ldots, r-1\}$.
We say that runner $t$ of $A$ is \emph{$r$-decomposable} if  it has positions
$\alpha_1 < \beta_1  < \cdots < \alpha_c < \beta_c$
such that, for each $k \in \{1,\ldots, c\}$,
position $\beta_k$ has a bead, positions $\alpha_k, \alpha_k+r,\ldots, \beta_k-r$
have gaps, and runner $t$ of $C$ is obtained by 
moving the bead in position~$\beta_k$ to the gap
in position $\alpha_k$.
\end{definition}


\begin{lemma}\label{lemma:dec}
Let the skew-partition $\lambda/\nu$ and the abacus $A$ be as defined.

\begin{thmlist}
\item 
Let $\beta$ be the greatest numbered position of $A$ that has a bead moved in a sequence of bead moves
leading to $\nu$. Then $\lambda / \nu$ has a final $r$ border-strip if and only if
$A$ has a gap in position $\beta - r$.

\item The skew-partition $\lambda / \nu$ is $r$-decomposable if and only if runner $t$ of $A$
is $r$-decomposable for each $t \in \{0,\ldots, r-1\}$.
\end{thmlist}
\end{lemma}

\begin{proof}
Let $d$ be minimal such that $\lambda_d > \nu_d$.
The bead in position $\beta$ corresponds to the box in position $(d,\lambda_d)$ of $[\lambda]$. By Lemma~\ref{lemma:basic}(i),
this box is the top-right box in an $r$ border-strip in $\lambda$ 
if and only if there is a gap in position $\beta - r$ of $A$. This proves~(i).
Part (ii) now follows by repeated applications of (i).
\end{proof}

%
%


The skew-partition $\lambda / \nu$ shown in Figure~1 is $2$-decomposable. It 
may be used to give an example of
Proposition~\ref{prop:anyorder} and Lemma~\ref{lemma:dec}.

\section{Proof of Equation~\eqref{eq:main}}

The proof is by induction on $m$. 
We begin with the  identity
\begin{equation}
\label{eq:Newton} mh_m = \sum_{\ell=1}^m p_\ell h_{m-\ell},
\end{equation}
which may be proved in a few lines working from the generating functions
$\sum_{m=0}^\infty h_mt^m = \prod_{i=1}^\infty (1-x_it)^{-1}$ and 
$\sum_{\ell=1}^\infty p_\ell t^\ell = \sum_{i=1}^\infty x_it(1-x_it)^{-1}$,
or found in \cite[I, Equation (2.11)]{MacDonald}.
The map $f \mapsto p_r \circ f$ is an endomorphism of the
ring of symmetric functions  (see \cite[I.8.6]{MacDonald}), so \eqref{eq:Newton} implies that
\begin{equation*} 
p_r \circ mh_m = p_r  \circ  \sum_{\ell=1}^m  p_\ell h_{m-\ell}  
= \sum_{\ell=1}^m (p_r \circ p_\ell)(p_r \circ h_{m-\ell}) 
= \sum_{\ell=1}^m p_{r\ell} (p_r \circ h_{m-\ell}).
\end{equation*}
Since $p_r \circ mh_m = mp_r \circ h_m$, it follows that
\begin{equation*} 
ms_\nu (p_r \circ h_m) = \sum_{\ell=1}^m s_\nu    (p_r \circ h_{m-\ell})p_{r \ell} .
\end{equation*}
By~\eqref{eq:MN} and induction we get
\begin{equation*}
ms_\nu (p_r \circ h_m) = \sum_{\ell=1}^m 
 \,\sum_{\mu \vvdash r(m- \ell) + |\nu|}\, \sum_{\lambda \vvdash r\ell + |\mu|}\,
\sgn_r(\mu / \nu) \sgn(\lambda/\mu) s_\lambda .
\end{equation*}
It is therefore sufficient to prove that if $\lambda / \nu$ is a skew-partition of $rm$ then
\begin{equation}
\label{eq:sgn} 
m\sgn_r(\lambda / \nu)   = \sum_\mu   \sgn (\lambda / \mu) \sgn_r(\mu / \nu)
\end{equation}
where the sum is over all partitions $\mu$ such that $\lambda / \mu$ is a border-strip
of length divisible by $r$ and $\mu/\nu$ is a skew-partition.


Fix an $r$-runner abacus $A$ for $\lambda$.
%
We may assume that one side of~\eqref{eq:sgn} is non-zero, and so
an abacus $C$ for $\nu$ 
can be obtained by a  sequence of bead moves on the runners of $A$.
We say that a runner of $A$ is of \emph{type}

\begin{enumerate}
\item[(I)] if it is $r$-decomposable (see Definition~\ref{defn:rdec});

\item[(II)] if it is not $r$-decomposable but an $r$-decomposable runner can be obtained
by swapping a bead on this runner with a gap one or more positions above~it;

\item[(III)] if it is neither of type (I) nor of type (II).
\end{enumerate}
For example, in Figure~2 after the proof of 
Proposition~\ref{prop:II}, runner $0$ of $A$ 
has type~(II) and runner $1$ has type (I).



By Lemma~\ref{lemma:basic}(ii),
swapping a bead and a gap as described in (II) corresponds to removing a border-strip
of length divisible by $r$ from $\lambda$ to leave a partition $\mu$. 
The corresponding contribution of $\sgn(\lambda/\mu)\sgn_r(\mu/\nu)$ 
to the right-hand side of~\eqref{eq:sgn} is non-zero if
and only if $\mu/\nu$ is an $r$-decomposable skew-partition. Hence
if~$A$ has a runner of type (III) or two or more runners
of type (II), then both sides of~\eqref{eq:sgn} are zero. The following
two propositions deal with the remaining cases. In both cases an
example is given following the proof.

\begin{proposition}\label{prop:I}
If all runners of $A$ have type \emph{(I)} then~\eqref{eq:sgn} holds.
\end{proposition}

\begin{proof}
Let $\mu$ be a partition such that $\lambda / \mu$ is a border-strip
of length divisible by $r$ and $\mu/\nu$ is a skew-partition. Suppose that
$\mu$ is obtained by moving a bead $b$ on runner $t$ of $\lambda$.
By Lemma~\ref{lemma:dec}(ii) there are positions
$\alpha_1 < \beta_1  < \cdots < \alpha_c < \beta_c$
on runner $t$, such that, for each $k \in \{1,\ldots, c\}$,
position $\beta_k$ has a bead, positions $\alpha_k, \alpha_k+r,\ldots, \beta_k-r$
have gaps, and $\nu$ is obtained by moving the bead in position~$\beta_k$ to the gap
in position $\alpha_k$.
Since $\mu/\nu$ is a skew-partition, the
bead $b$ must be in one of the positions $\beta_k$ and, after the move giving~$\mu$, it must 
be in position $\beta_k - rq$ for some $q$ such that $1 \le q \le (\beta_k-\alpha_k)/r$.
Since there are gaps in $A$ in positions $\beta_k-rq, \ldots, \beta_k-r$, this move can
also be achieved by a sequence of~$q$ single-step bead moves of~$b$.
Since there are gaps in $A$ in positions $\alpha_k,\ldots, \beta_k-(r+1)q$, the
runner obtained after moving bead $b$ 
is still $r$-decomposable.
Hence, by Lemma~\ref{lemma:dec}, $\mu/\nu$ is $r$-decomposable.
By Proposition~\ref{prop:anyorder}, noting that bead $b$ can be moved 
from position $\beta_k-rq$ to position $\alpha_k$ by  single-step bead moves,
we have $\sgn(\lambda/\mu)\sgn_r(\mu/\nu) = \sgn_r(\lambda/\nu)$. 

It follows that all the non-zero summands on the right-hand side of~\eqref{eq:sgn} are equal
to $\sgn_r(\lambda/\nu)$. 
The number of partitions~$\mu$ obtained by moving
a bead on runner $t$ that give a non-zero summand
is $(\beta_1-\alpha_1)/r + \cdots + (\beta_c - \alpha_c)/r$. Summing
over all runners, and using that $\lambda / \nu$ is a skew-partition of $rm$,
we see that there are exactly $m$ non-zero summands. This completes the proof.
\end{proof}

For an example consider the skew-partition 
$\lambda / \nu$ and the border-strip $\lambda /\mu$
shown in Figure~1. The partition $\mu$ is obtained by moving
the bead in position $15$ to position $5$: we denote this move by $(15,5)$.
The final $2$~border-strips removed from $\mu$ to obtain $\nu$ correspond to the bead moves
$(19,17), (14,12), (5,3), (4,2), (2,0)$ and have
heights $0, 0, 1,1,1$. Since $\lambda / \mu$ has height $3$, we have
$\sgn(\lambda/\mu)\sgn_2(\mu /\nu)$ $= (-1)^6 = 1$. The 
final $2$ border-strips removed from $\lambda$ to obtain $\nu$
correspond to the bead moves
$(19,17)$, $(15,13)$, $(14,12)$, $(13,11)$, $(11,9)$, $(9,7)$, 
$(7,5)$, $(5,3)$, $(4,2)$, $(2,0)$
and have heights $0, 1, 1, 1,0,1,1$, $1$, $1$, $1$,
so $\sgn_2(\lambda / \nu) = (-1)^8 = 1$.

\begin{proposition}\label{prop:II}
If there is a unique runner of $A$ of type \emph{(II)} and all other runners
have type \emph{(I)} then both sides of~\eqref{eq:sgn} are zero.
\end{proposition}

\begin{proof}
By Lemma~\ref{lemma:dec}(ii), 
$\lambda / \nu$ is not $r$-decomposable. Hence the left-hand side of~\eqref{eq:sgn}
is zero. Let runner $t$ be the unique runner of $A$ of type (II).
Since runner $t$ is not $r$-decomposable, 
there are beads $d$ and $d^\star$ on this runner, 
in positions $\delta$ and $\delta^\star$ respectively,
such that $\delta > \delta^\star$ and in any sequence of single-step bead moves
leading from $A$ to $C$, bead~$d$ 
finishes above
position~$\delta^\star$. Choose $\delta$ maximal with this property. Then
there are positions 
\[ \alpha_0 < \alpha_1 \le \beta_1 < \alpha_2 \le \beta_2 < 
\cdots < \alpha_b \le \beta_{b} < \beta_{b+1} \]
on runner $t$
such that (a) $\beta_{b} = \delta^\star$ and $\beta_{b+1} = \delta$, 
(b) the beads between positions $\alpha_0$ and $\beta_{b+1}$ are exactly those
in positions $\beta_1, \ldots, \beta_{b+1}$,
(c) in~any sequence of single-step bead moves leading from $A$ to $C$,
for each $k \in \{1,\ldots, b+1\}$, 
the bead in position $\beta_k$ finishes in the gap in position $\alpha_{k-1}$,
and~(d) swapping bead $d$ with the gap in position $\alpha_0$ gives
an $r$-decomposable runner.

Let~$P$ be the set of pairs $(\epsilon,\gamma)$ such that $\epsilon$ and $\gamma$
are positions on runner $t$ and the runner obtained
by swapping the bead in position $\epsilon$ 
with the gap in position $\gamma$ is $r$-decomposable. It follows from the
choice of $\delta$ that if $(\epsilon,\gamma) \in P$ then $\epsilon \in \{\delta^\star,
\delta\}$ and that $(\delta,\gamma) \in P$ if and only if $(\delta^\star,\gamma) \in P$.
Hence
\[ P = \bigl\{ (\epsilon, \gamma) : \epsilon \in \{\delta,\delta^\star\}, \
\gamma \in \{\alpha_0, \ldots, \alpha_1-1\} \bigr\}. \]

Let $(\delta,\gamma) \in P$, let $B$ be the abacus obtained by swapping bead $d$ with the gap in position $\gamma$, and
let $\mu$ be the partition represented by $B$. Define~$B^\star$ and $\mu^\star$
analogously, replacing $d$ with $d^\star$. It suffices to show that
\begin{equation}
\label{eq:final}
\sgn(\lambda/\mu^\star)\sgn_r(\mu^\star/\nu) = -\sgn(\lambda/\mu)\sgn_r(\mu/\nu)
\end{equation}
so the contributions from $\mu$ and $\mu^\star$ to~\eqref{eq:sgn} cancel.
We do this using Lemma~\ref{lemma:signs} and
a sign reversing pairing on sequences of bead moves from $A$ to $C$.
This pairing is illustrated in Figure~2 and in the example following this proof.


Fix a sequence of bead moves that first swaps bead 
$d$ with the gap in position~$\gamma$ (giving~$B$) then makes
single-step bead moves to go from $B$ to~$C$. This sequence is paired with the sequence that
first swaps bead $d^\star$ with the gap in position~$\gamma$ 
(giving~$B^\star$), then moves bead $d$
to position $\delta^\star$ by  single-step moves
(giving $B$,  with beads $d$ and $d^\star$ swapped compared to the first sequence)
and then makes the same sequence of single-step moves to go from $B$ to $C$.
Let $\mathcal{J}$ and $\mathcal{J}^\star$ be the set of pairs $\{ \beta,\beta' \}$
defined, as in Lemma~\ref{lemma:signs}, for these two sequences of bead moves.
It is clear that $\mathcal{J}$ and $\mathcal{J}^\star$ agree
except for pairs involving the positions $\delta$ and $\delta^\star$. Moreover
$\{ \delta,\delta^\star \} \in \mathcal{J} \backslash \mathcal{J}^\star$.

Let $\alpha$ and $\alpha^\star$ be, respectively, 
the final positions of beads $d$ and $d^\star$ in $C$ after the sequence
of moves from $A$ to $B$ to $C$. (Equivalently, $\alpha$ and $\alpha^\star$ are, respectively, the final position
of beads $d^\star$ and $d$ in $C$, after the sequence of moves from $A$ to $B^\star$ to $B$ to $C$.) 
Let $\mathcal{A}$ be the set of positions of $A$ that have a bead, excluding positions $\delta$ and $\delta^\star$.
For $\beta \in \mathcal{A}$, let $\bar{\beta}$ 
be the final position in $C$, after either sequence of moves, 
of the bead starting in position $\beta$ of $A$.

The following four claims are routine to check:
\begin{align*}
\text{$\{\beta,\delta\} \in \mathcal{J}$ and $\{\beta,\delta^\star\} \not\in \mathcal{J}^\star$}
&\iff \text{$\delta^\star < \beta < \delta$ and $\alpha < \bar{\beta}$,} \\
\text{$\{\beta,\delta^\star\} \in \mathcal{J}$ and $\{\beta,\delta\} \not\in \mathcal{J}^\star$}
&\iff \text{$\delta^\star < \beta < \delta$ and $\bar{\beta} < \alpha^\star$,} \\
\text{$\{\beta,\delta\} \not\in \mathcal{J}$ and $\{\beta,\delta^\star\} \in \mathcal{J}^\star$}
&\iff \text{$\delta^\star < \beta < \delta$ and $\bar{\beta} < \alpha$,} \\
\text{$\{\beta,\delta^\star\} \not\in \mathcal{J}$ and $\{\beta,\delta\} \in \mathcal{J}^\star$}
&\iff \text{$\delta^\star < \beta < \delta$ and $\alpha^\star < \bar{\beta}$.}
\end{align*}
Let $X_\mathcal{J}$, $Y_\mathcal{J}$, $X_{\mathcal{J}^\star}$, $Y_{\mathcal{J}^\star}$ be the sets
of $\beta \in \mathcal{A}$ satisfying each of these conditions, respectively. These sets are obstacles
to a bijection $\mathcal{J} \backslash \big\{ \{\delta,\delta^\star \} \bigr\} \longleftrightarrow
\mathcal{J}^\star$ defined by
$\{\beta,\delta\} \longleftrightarrow \{\beta,\delta^\star\}$.
Observe that
\begin{align*} 
X_\mathcal{J} &= \{ \beta \in \mathcal{A} : \delta^\star < \beta < \delta,\ \alpha < \bar{\beta} < \alpha^\star \} \cup Y_{\mathcal{J}^\star},  \\
Y_\mathcal{J} &= \{ \beta \in \mathcal{A} : \delta^\star < \beta < \delta,\ \alpha < \bar{\beta} < \alpha^\star \} \cup X_{\mathcal{J}^\star}. 
\end{align*}
It follows that
\[ |\mathcal{J}| = 2\bigl| 
\{ \beta \in \mathcal{A} : \delta^\star < \beta < \delta,\ \alpha < \bar{\beta} < \alpha^\star \} \bigr|
+ |\mathcal{J}^\star| + 1
\]
where the final summand comes from $\{\delta, \delta^\star\}$.
Hence  $|\mathcal{J}|$ and $|\mathcal{J}^\star|$ have opposite parities. Equation~\eqref{eq:final}
now follows from Lemma~\ref{lemma:signs}. This completes the proof.
\end{proof}

\begin{figure}[b]
\hspace*{-0.2in}\scalebox{1}{\includegraphics{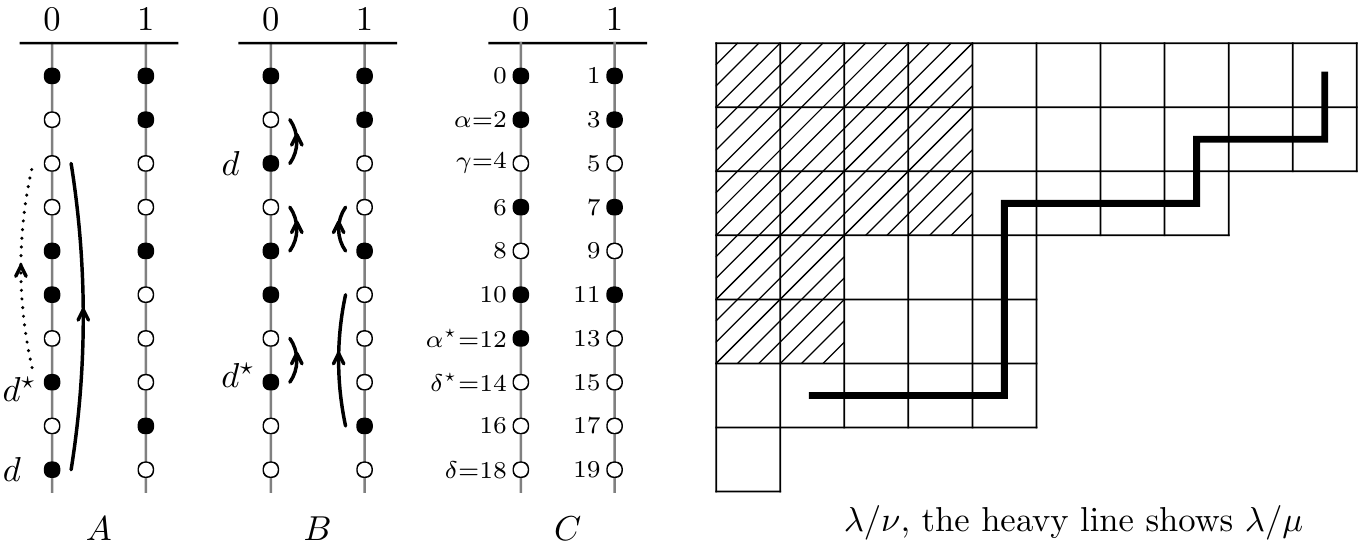}}

\vspace*{0.1in}
\hspace*{-0.2in}\scalebox{1}{\includegraphics{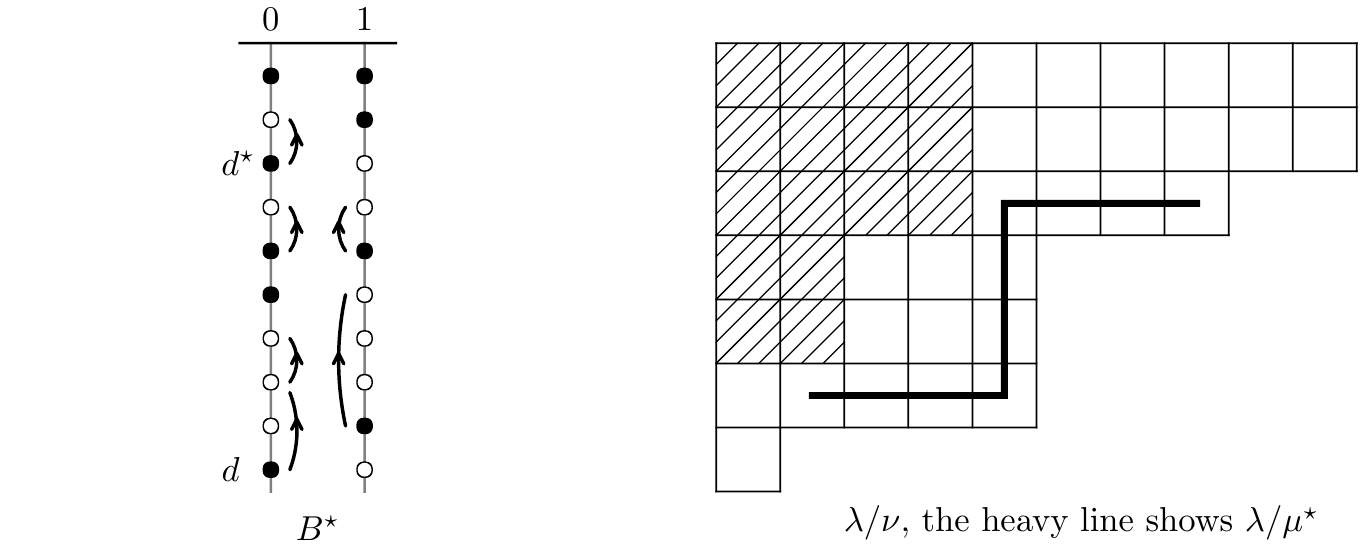}}
\caption{\small Example to illustrate Proposition~\ref{prop:II} when 
$r=2$ and $\lambda/\nu = (10,10,8,5,5,5,1) / (4,4,4,2,2)$. The partitions $\mu$ and $\mu^\star$ are $(9,7,4,4,4,1,1)$ and $(10,10,4,4,4,1,1)$.
Abaci
$A$, $B$, $B^\star$ and $C$ for $\lambda$, $\mu$, $\mu^\star$, $\nu$, respectively, 
are shown. The 
bead moves between these abaci are indicated by arrows: $B$ is obtained
from $A$ by the move $(\delta,\gamma) = (18,4)$ shown by a solid arrow and $B^\star$
is obtained from $A$ by the move $(\delta^\star,\gamma) = (14,4)$
shown by a dotted arrow.}
\end{figure}

In the example shown in Figure~2 below with $r=2$, we have 
$\delta = 18$, $\delta^\star = 14$,
$\alpha^\star = 12$, $\alpha = 2$ and $\gamma = 4$. The
set $P$ is $\{(\delta,2), (\delta^\star,2), (\delta,4), (\delta^\star, 4)\}$. 
The sets $\mathcal{J}$ and $\mathcal{J}^\star$ are
\begin{align*}
\mathcal{J} &= \bigl\{ \{10,18\}, \{9,18\}, \{8,18\}, \{3,18\} \bigr\} \cup
\bigl\{ \{17,18\}, \{14,17\}  \bigr\} \cup \bigl\{ \{14,18\} \bigr\}, \\
\mathcal{J}^\star &= \bigl\{ \{10,14\}, \{9,14\}, \{8,14\}, \{3,14\} \bigr\}.
\end{align*}
The second set in the union for $\mathcal{J}$ gives the pairs coming from
$X_\mathcal{J} = Y_\mathcal{J} = 
 \{ \beta \in \mathcal{A} :
\delta^\star < \beta < \delta, \ \alpha < \bar{\beta} < \alpha^\star \} = \{17\}$.
In this example $X_{\mathcal{J}^\star} = Y_{\mathcal{J}^\star} = \varnothing$. We have
$\sgn(\lambda/  \mu^\star) \sgn_2(\mu^\star / \nu) = 1  = -
\sgn(\lambda / \mu)\sgn_2(\mu/\nu)$ 
as predicted by~\eqref{eq:sgn}.

\section*{Acknowledgements}

The author thanks an anonymous referee for
exceptionally useful and detailed comments on an earlier version of this paper. He also
thanks Anton Evseev and Rowena Paget for their comments and many helpful and stimulating discussions.

\def\cprime{$'$} \def\Dbar{\leavevmode\lower.6ex\hbox to 0pt{\hskip-.23ex
  \accent"16\hss}D} \def\cprime{$'$}
\providecommand{\bysame}{\leavevmode\hbox to3em{\hrulefill}\thinspace}
\providecommand{\MR}{\relax\ifhmode\unskip\space\fi MR }
\providecommand{\MRhref}[2]{%
  \href{http://www.ams.org/mathscinet-getitem?mr=#1}{#2}
}
\providecommand{\href}[2]{#2}

\end{document}